\documentclass[10pt,twocolumn,twoside]{IEEEtran}

\usepackage{amssymb,amsfonts,amsmath,amsthm}
\usepackage{mathtools}
\usepackage{eucal,bm}
\usepackage{gensymb,subfigure}
\usepackage{multirow}
\usepackage{booktabs}
\usepackage{times,color,url}
\usepackage{comment}
\usepackage{cite}
\usepackage{graphicx}
\usepackage[all]{xy}
\usepackage{algorithm}
\usepackage{algpseudocode}
\usepackage{subfigure}
\usepackage{bm}
\usepackage{color}
\usepackage{wrapfig}
\usepackage{epsf}
\usepackage{times,mathptm}
\usepackage{url}
\usepackage{nomencl}
\newtheorem{theorem}{Theorem}

\newtheorem{lemma}{Lemma}
\DeclareMathOperator*{\argmin}{arg\,min}
\newcommand{\squeezeup}{\vspace{-2.5mm}}
\begin{document}
\title{Learning Topology of the Power Distribution Grid with and without Missing Data}
\author{\IEEEauthorblockN{Deepjyoti~Deka*, Scott~Backhaus\dag, and Michael~Chertkov\dag\\}
\IEEEauthorblockA{*Electrical \& Computer Engineering, University of Texas at Austin,
\dag Los Alamos National Laboratory, USA\\
Email: deepjyotideka@utexas.edu, backhaus@lanl.gov, chertkov@lanl.gov}}

\maketitle

\begin{abstract}
Distribution grids refer to the part of the power grid that delivers electricity from substations to the loads. Structurally a distribution grid is operated in one of several radial/tree-like topologies that are derived from an original loopy grid graph by opening switches on some lines. Due to limited presence of real-time switch monitoring devices, the operating structure needs to be estimated indirectly. This paper presents a new learning algorithm that uses only nodal voltage measurements to determine the operational radial structure. The algorithm is based on the key result stating that the correct operating structure is the optimal solution of the minimum-weight spanning tree problem over the original loopy graph where weights on all permissible edges/lines (open or closed) is the variance of nodal voltage difference at the edge ends. Compared to existing work, this spanning tree based approach has significantly lower complexity as it does not require information on line parameters. Further, a modified learning algorithm is developed for cases when the input voltage measurements are limited to only a subset of the total grid nodes. Performance of the algorithms (with and without missing data) is demonstrated by experiments on test cases.
\end{abstract}

\section{Introduction}
\label{sec:intro}
Distribution grids constitute the low voltage segment of the power system delivering electricity from substations to end-users. Both structurally and operationally the distribution grids are distinct from the transmission (high voltage) portion of the power system. A typical distribution grid is operated  as a collection of disjoint tree graphs, each growing from substations at the root to customers. However, the complete layout of the distribution system is loopy to allow multiple alternatives for the trees to energize operationally. Switching from one layout to another, implemented through switch on/off devices placed on many segments of the distribution grid \cite{distgridpart1}, can take place rather often, in some cases few times an hour. (See Fig.~\ref{fig:city} for the illustration.)
More frequent reconfiguration of the distribution is also promoted by recent in-mass integration of smart meters, PMUs \cite{phadke1993synchronized} and smart devices, such as deferrable loads and energy storage devices. Mixed operational responsibilities in monitoring and operations, as well as the growing role of the new smart devices and controls, make fast and reliable estimation of the operational configuration of the distribution grid an important practical task, complicated by the lack of real-time, line-based measurements. In such a scenario, to estimate the distribution grid operational topology one ought to rely only on nodal measurements of voltage and end-user consumption.  Notice, that brute force (combinatorial) check of topologies for the nodal measurement consistency is prohibitively expensive with the complexity growing exponentially with the number of loops in the grid layout.

In this work we focus on beating the naive exponential complexity of the operational topology learning task by exploring power flow specific correlations between available nodal measurements. In particular, \emph{we develop a spanning tree algorithm that reconstructs the radial operational topology from the original loopy layout by using functions of nodal voltage magnitudes as edge weights.} Computational complexity of this algorithm is order $O(n\log n)$ in the size of the loopy graph's edge set. Moreover, the algorithm is generalized to the case when some nodes are hidden.

\subsection{Prior Work}

Several approaches in the past have been made to learn the topology of power grids under different operating conditions and available measurements. \cite{he2011dependency} uses a Markov random field model for bus phase angles to build a dependency graph to identify faults in the grids. \cite{bolognani2013identification} presents a topology identification algorithm for distribution grids that uses the signs of elements in the inverse covariance matrix of voltage measurements. \cite{berkeley} compares available time-series observations from smart meters with a database of permissible signatures to identify topology changes. This is similar to envelope comparison schemes used in parameter estimation \cite{sandia1, sandia2}. For available line flow measurements, topology estimation using maximum likelihood tests was analyzed in \cite{ramstanford}. In our own prior work \cite{distgridpart1,distgridpart2}, we analyzed an iterative greedy structure learning algorithm using trends in second order moments of voltages. \cite{distgridpart2} also presented the first attempt at topology learning from incomplete voltage data where nodes with missing voltages are separated by greater than two hops. The aforementioned approaches are specific to power grid graphs and typically not linked to research in probabilistic Graphical Models (GM) \cite{wainwright2008graphical} used to study statistics of images, languages, social networks, and communication schemes. Learning generic (loopy) structures from pair-wise correlations in a GM is a difficult task, normally based on the maximal likelihood  \cite{wainwright2008graphical} with regularization for sparsity \cite{ravikumar2010high} and  greedy schemes utilizing conditional mutual information \cite{anandkumar2011high, netrapalli2010greedy}. However, the GM-based learning simplifies dramatically when used, following the famous Chow-Liu approach \cite{chow1968approximating}, to reconstruct the spanning tree maximizing edge-factorized mutual information. \cite{choi2011learning} generalizes this technique to learn tree structured GMs with latent variables (missing data) using information distances as edge weights. 

\subsection{Contribution of This Work}
Following \cite{distgridpart1,distgridpart2}, we consider linear lossless AC power flow models (also called, following \cite{89BWa,89BWb} Lin-Dist-Flow) and assume that fluctuations of consumption at the nodes are uncorrelated. In this setting, our main result states that reconstruction of the operating grid topology is equivalent to solving the minimum weight spanning tree problem defined over the loopy graph of the grid layout where edge weights are given by \emph{variances in voltage magnitude differences across the edges}. We use this result to formulate the operating topology as a spanning tree reconstruction problem that needs only empirical voltage magnitude measurements as input. As spanning trees can be efficiently reconstructed, our learning algorithm has much lower average and worst-case computational complexity compared to existing techniques \cite{bolognani2013identification,distgridpart2}. While our algorithm does not require knowing line impedances, these can be used to estimate additionally statics of power consumption. Further, we extend the topology learning algorithm to the case with missing voltage data. The extension works provided nodes with missing data are separated by at least two hops from each other and  covariances of nodal power consumption are available. Compared to our prior work \cite{distgridpart2} on learning with missing data, the spanning tree approach has lower complexity. It also allows extension to cases with lesser restrictions on missing data. Our algorithm shows some commonality with the GM based spanning-tree learning of \cite{choi2011learning}. However the key difference is that our approach relies principally on the Kirchoff's laws of physical network flows contrary to the measure of conditional independence utilized in \cite{chow1968approximating,choi2011learning}. Thus, voltage magnitude based edge weights used in our work are not restricted to satisfy graph additivity unlike information distances in GM. Further, in the case with missing data, we use power flow relations between nodal voltages and injections that, to the best of our knowledge, do not have an analog in GM learning literature. We highlight the performance of our algorithm through experiments on test distribution grids for both cases, with or without missing data.

The rest of the manuscript is organized as follows. Section \ref{sec:structure} introduces notations, nomenclature and power flow relations in the distribution grids. Section \ref{sec:trends} describes important features of the  nodal voltage magnitudes. This Section also contains the proof of our main -- spanning tree learning/reconstruction -- theorem. Algorithm reconstructing operational spanning tree in the case of complete visibility (voltage magnitudes are observed at all nodes) is discussed in Section \ref{sec:algo1}. Modification of the algorithm which allows for some missing data (at the nodes separated by at least two hopes) is described in Section \ref{sec:missing}. This Section also contains a brief discussion of some other extensions/applications of our approach. Simulation results of our learning algorithm on a test radial network are presented in Section \ref{sec:experiments}. Finally, Section \ref{sec:conclusions} contains conclusions and discussion of future work.

\section{Distribution Grid: Structure and Power Flows}
\label{sec:structure}

\textbf{Radial Structure}: The original distribution grid is denoted by the graph ${\cal G}=({\cal V},{\cal E})$, where ${\cal V}$ is the set of buses/nodes of the graph and ${\cal E}$ is the set of all undirected lines/edges (open or operational). We denote nodes by alphabets ($a$, $b$,...) and the edge connecting nodes $a$ and $b$ by $(ab)$. The operational grid has a `radial' structure as shown in Fig.~\ref{fig:city}. In general, the operational grid is a collection of $K$ disjoint trees, $\cup_{i=1,\cdots,K}{\cal T}_i$ where each tree's root node has degree one (connected by one edge) and represents a substation.
\begin{figure}[!bt]
\centering
\includegraphics[width=0.38\textwidth, height =.29\textwidth]{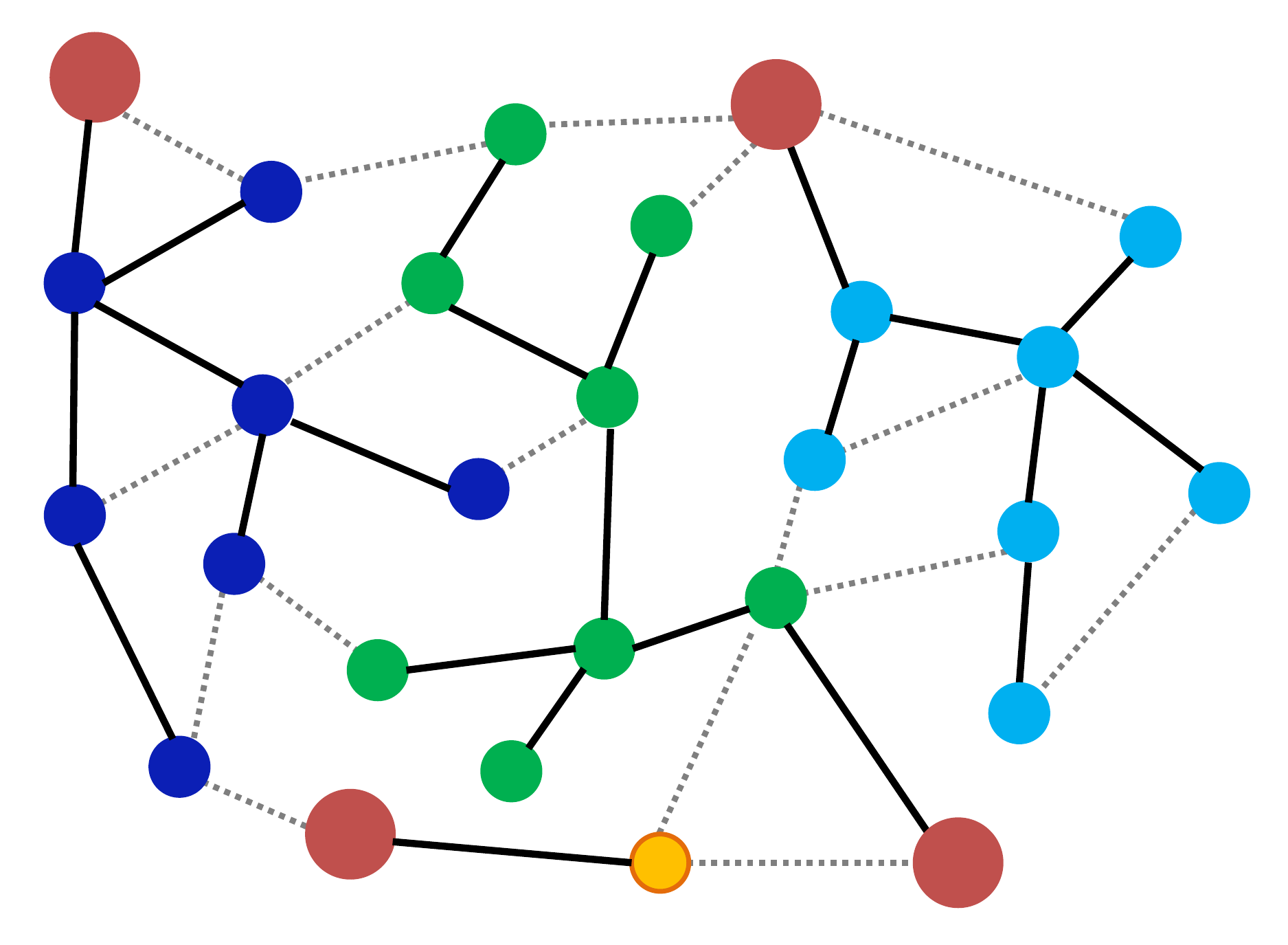}
\caption{A distribution grid with $4$ substations (large red nodes). The operational radial trees are formed by solid lines (black). Dotted grey lines represent open switches. Non-substation nodes within each tree are marked with the same color.
\label{fig:city}}
\end{figure}
In this paper, we will mainly focus on grids where the operational structure consists of only one tree $\cal T$ with nodes ${\cal V}_{\cal T}$ and operational edge set ${\cal E}_{\cal T} \subset {\cal E}$. Generalization to the case with multiple disjoint trees will be discussed along side major results.

\textbf{Power Flow (PF) Models}: Let $z_{ab}=r_{ab}+i x_{ab}$ denote the complex impedances of a line $(ab)$ ($i^2=-1$). Here $r_{ab}$ and $x_{ab}$ are  line resistance and reactance respectively. Kirchhoff's laws express the complex valued power injection at a node $a$ in tree ${\cal T}$ as
\begin{align}
P_a =p_a+i q_a
= \underset{b:(ab)\in{\cal E}_{\cal T}}{\sum}\frac{v_a^2-v_a v_b\exp(i\theta_a-i\theta_b)}{z_{ab}^*}\label{P-complex1}
\end{align}
where the real valued scalars, $v_a$, $\theta_a$, $p_a$ and $q_a$ denote the voltage magnitude, voltage phase, active and reactive power injection respectively at node $a$. $V_a (= v_a\exp(i\theta_a))$ and $P_a$ denote the nodal complex voltage and injection respectively.
One node (substation/root node in our case) is considered as reference and the voltage magnitude and phase at every non-substation node are measured relative to the reference values. As the complex power injection at the reference bus is given by negation of the sum of injections at other buses, without a loss of generality the analysis can be limited to a reduced system, where one ignores reference substation bus voltages and power injections. Under realistic assumption that losses of both active and reactive power in lines of a distribution system are small, Eq.~(\ref{P-complex1}) can be linearized as follows.

\textbf{Linear Coupled (LC) model} \cite{distgridpart1,distgridpart2}: In this model, phase difference between neighboring nodes and magnitude deviations ($v_a -1=\varepsilon_a$) from the reference voltage are assumed to be small. The PF Eqs.~(\ref{P-complex1}) are linearized jointly over both voltage magnitude and phase to give:
\begin{align}
\varepsilon = H^{-1}_{1/r}p + H^{-1}_{1/x}q~~ \theta = H^{-1}_{1/x}p - H^{-1}_{1/r}q  \label{PF_LPV_p}
\end{align}
Here, $p,q,\varepsilon$ and $\theta$ are the vectors of real power, reactive power, voltage magnitude deviation and phase angle respectively at the non-substation nodes of the reduced system. $H_{1/r}$ and $H_{1/x}$ denote the reduced weighted Laplacian matrices for $\cal T$ where  reciprocal of resistances and reactances are used respectively as edge weights. The reduced Laplacian matrices are of full rank and constructed by removing the row and column corresponding to the reference bus from the true Laplacian matrix.

\cite{distgridpart1} shows that the LC-PF model is equivalent to the LinDistFlow model \cite{89BWa,89BWb,89BWc}, if deviations in voltage magnitude are assumed to be small and thus ignored. (Notice, that if line resistances are equated to zero, the LC-PF model reduces to the DC PF model \cite{abur2004power} used for transmission grids.) We can express means $(\mu_{\theta}, \mu_{\varepsilon})$ and covariance matrices $(\Omega_{\varepsilon}, \Omega_{\theta}, \Omega_{\theta\varepsilon})$ of voltage magnitude deviations and phase angles in terms of corresponding statistics of power injections using Eq.~(\ref{PF_LPV_p}) as shown below. Other quantities can be similarly determined.
\begin{align}
\mu_{\theta} &= H^{-1}_{1/x}\mu_p - H^{-1}_{1/r}\mu_q,~~\mu_\varepsilon = H^{-1}_{1/r}\mu_p + H^{-1}_{1/x}\mu_q\label{means}\\
\Omega_{\varepsilon} &= H^{-1}_{1/r}\Omega_{p}H^{-1}_{1/r} + H^{-1}_{1/x}\Omega_qH^{-1}_{1/x}+H^{-1}_{1/r}\Omega_{pq}H^{-1}_{1/x}\nonumber\\
&~+H^{-1}_{1/x}\Omega_{qp}H^{-1}_{1/r}\label{volcovar1}
\end{align}

In the next Section, we derive key results for functions of nodal voltages in a radial distribution grid that will subsequently be used in the topology learning algorithm.

\section{Properties of Voltage Magnitudes in Radial Grids}
\label{sec:trends}

Consider grid tree $\cal T$ with operational edge set ${\cal E}_{\cal T}$. Let ${\cal P}^{a}_{\cal T}$ denote the set of edges in the unique path from node $a$ to the root node (reference bus) in tree ${\cal T}$. A node $b$ is termed as a descendant of node $a$ if ${\cal P}^{b}_{\cal T}$ includes some edge $(ac)$ connected to node $a$. We use $D^{a}_{\cal T}$ to denote the set of descendants of $a$. By definition, $a \in D^{a}_{\cal T}$. If $b$ is an immediate descendant of $a$ ($(ab) \in {\cal E}_{\cal T}$), we term $a$ as parent and $b$ as its child. These definitions are illustrated in Fig \ref{fig:picHinv}.
\begin{figure}[!bt]
\centering
\includegraphics[width=0.24\textwidth,height=.23\textwidth]{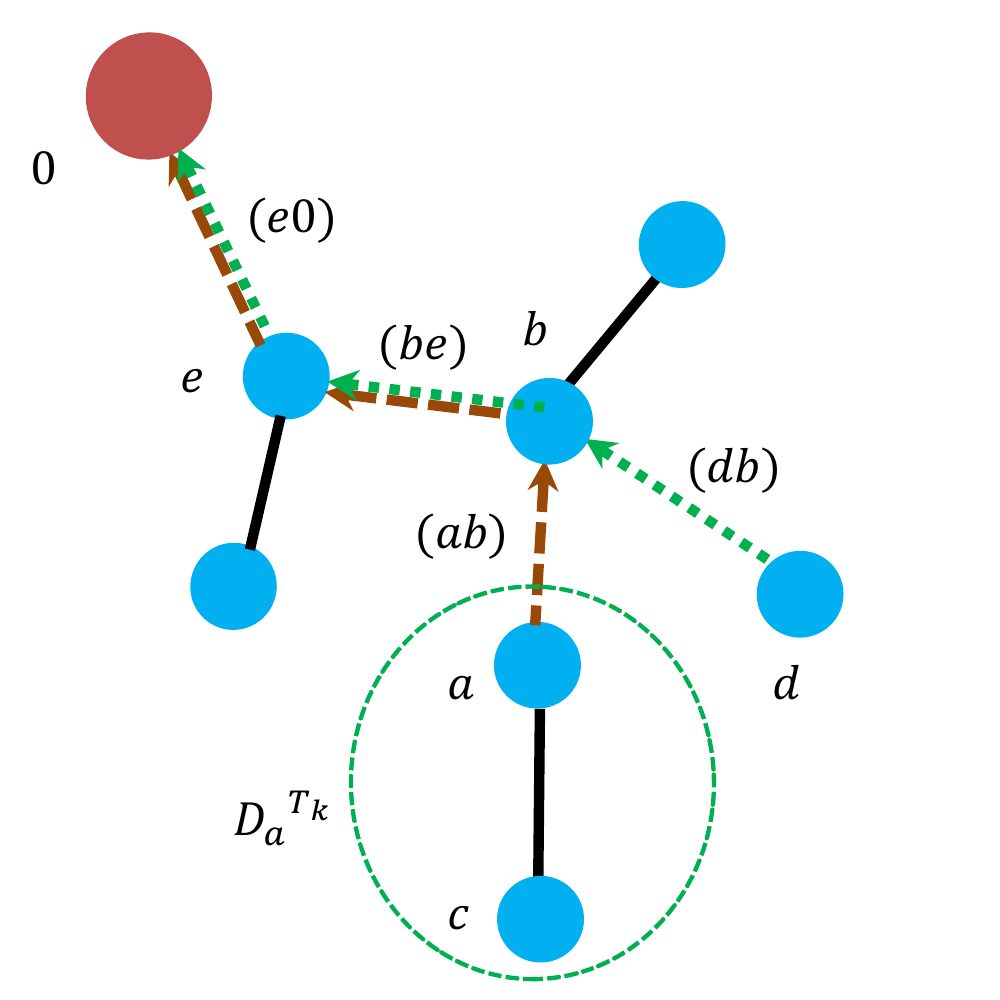}
\squeezeup
\caption{The Figure shows distribution grid tree with substation/root node colored in red. Here, nodes $a$ and $c$ are descendants of node $a$. Dotted lines represent the paths from nodes $a$ and $d$ to the root node. The paths' common edges give $H_{1/r}^{-1}(a,d) = r_{be}+ r_{e0}$.
\label{fig:picHinv}}
\end{figure}

Due to the radial topology of $\cal T$, the inverse of the reduced weighted graph Laplacian matrix $H_{1/r}$ has the following structure (see Section $4$ in \cite{distgridpart1} for details).
\begin{align}
 H_{1/r}^{-1}(a,b)&= \sum_{(cd) \in {\cal P}^a_{\cal T}\bigcap {\cal P}^b_{\cal T}} r_{cd} \label{Hrxinv}
\end{align}
Thus, the $(a,b)^{th}$ entry in $H^{-1}_{1/r}$ is given by the sum of line resistances of edges that are included in the path to the root from either node as shown in Fig.~\ref{fig:picHinv}. For nodes $a$ and its parent $b$ in tree ${\cal T}$ (see Fig.~\ref{fig:picHinv}), it follows from Eq.~(\ref{Hrxinv}) that
\begin{align}
{\huge H}_{1/r}^{-1}(a,c)-{\huge H}_{1/r}^{-1}(b,c) &&=\begin{cases}r_{ab} & \quad\text{if node $c \in D^a_{\cal T}$}\\
0 & \quad\text{otherwise,} \end{cases} \label{Hdiff}
\end{align}
We use Eqs.~(\ref{Hrxinv}) and (\ref{Hdiff}) to prove our results on voltage magnitude relations. The results hold under the following assumptions.

\textbf{Assumption $1$:} Power Injection at different nodes are not correlated, while active and reactive injections at the same node are positively correlated. Mathematically, $\forall a,b$ non-substation nodes
\begin{align}
\Omega_{qp}(a,a) > 0,~\Omega_p(a,b) = \Omega_q(a,b)= \Omega_{qp}(a,b) = 0 \nonumber
\end{align}
Note that this is a valid assumption for many distribution grids due to independence between different nodal load fluctuations and alignment/correlations between same node's active and reactive power usage.

Under Assumption $1$, we state the following result without proof. (See \cite{distgridpart2} for details.)
\begin{theorem}\label{Theorem1_LC} \cite[Theorem 1]{distgridpart2}
If node $a \neq b$ is a descendant of node $b$ on tree ${\cal T}$ then $\Omega_{\varepsilon}(a,a) > \Omega_{\varepsilon}(b,b)$.
\end{theorem}

Next, we define the term $\phi_{ab} = \mathbb{E}[(\varepsilon_a - \mu_{\varepsilon_a})-(\varepsilon_b-\mu_{\varepsilon_b})]^2 $, which is the variance of the difference in voltage magnitudes between nodes $a$ and $b$.
\begin{align}
\phi_{ab} = \Omega_{\varepsilon}(a,a) - 2\Omega_{\varepsilon}(a,b) + \Omega_{\varepsilon}(b,b) \label{expand}
\end{align}
where $\Omega_{\varepsilon}$ is given by Eq.~(\ref{volcovar1}). Expressing Eq.~(\ref{expand}) in terms of the four matrices that constitute $\Omega_{\varepsilon}$ and then using Eq.~(\ref{Hrxinv}) leads to the following expansion of $\phi_{ab}$  over power injections.
\begin{align}
&\phi_{ab} = \smashoperator[lr]{\sum_{d \in {\cal T}}}(H^{-1}_{1/r}(a,d)- H^{-1}_{1/r}(b,d))^2\Omega_p(d,d)\nonumber\\
&+(H^{-1}_{1/x}(a,d)- H^{-1}_{1/x}(b,d))^2 \Omega_q(d,d)+2\left(H^{-1}_{1/r}(a,d)- H^{-1}_{1/r}(b,d)\right)\nonumber\\
&\left(H^{-1}_{1/x}(a,d)- H^{-1}_{1/x}(b,d)\right)\Omega_{pq}(d,d) \label{usediff_1}
\end{align}

The next result identifies trends in $\phi_{ab}$ along the radial grid. Note that the first two cases in Lemma \ref{Lemmacases} are proven in \cite{distgridpart2}. The additional final case is opposite of the first case and helps develop our new learning scheme presented later in this paper.
\begin{lemma} \label{Lemmacases}
For three nodes $a \neq b \neq c$ in grid tree ${\cal T}$, $\phi_{ab} < \phi_{ac}$ holds for the following cases:
\begin{enumerate}
\item Node $a$ is a descendant of node $b$ and node $b$ is a descendant of node $c$ (see Fig.~\ref{fig:item1}).
\item Nodes $a$ and $c$ are descendants of node $b$ and the path from $a$ to $c$ passes through node $b$ (see Fig.~\ref{fig:item2}).
\item Nodes $c$ is a descendant of node $b$ and node $b$ is a descendant of node $a$ (see Fig.~\ref{fig:item3}).
\end{enumerate}
\end{lemma}
\begin{figure}[!bt]
\centering
\hspace*{\fill}
\subfigure[]{\includegraphics[width=0.1625\textwidth]{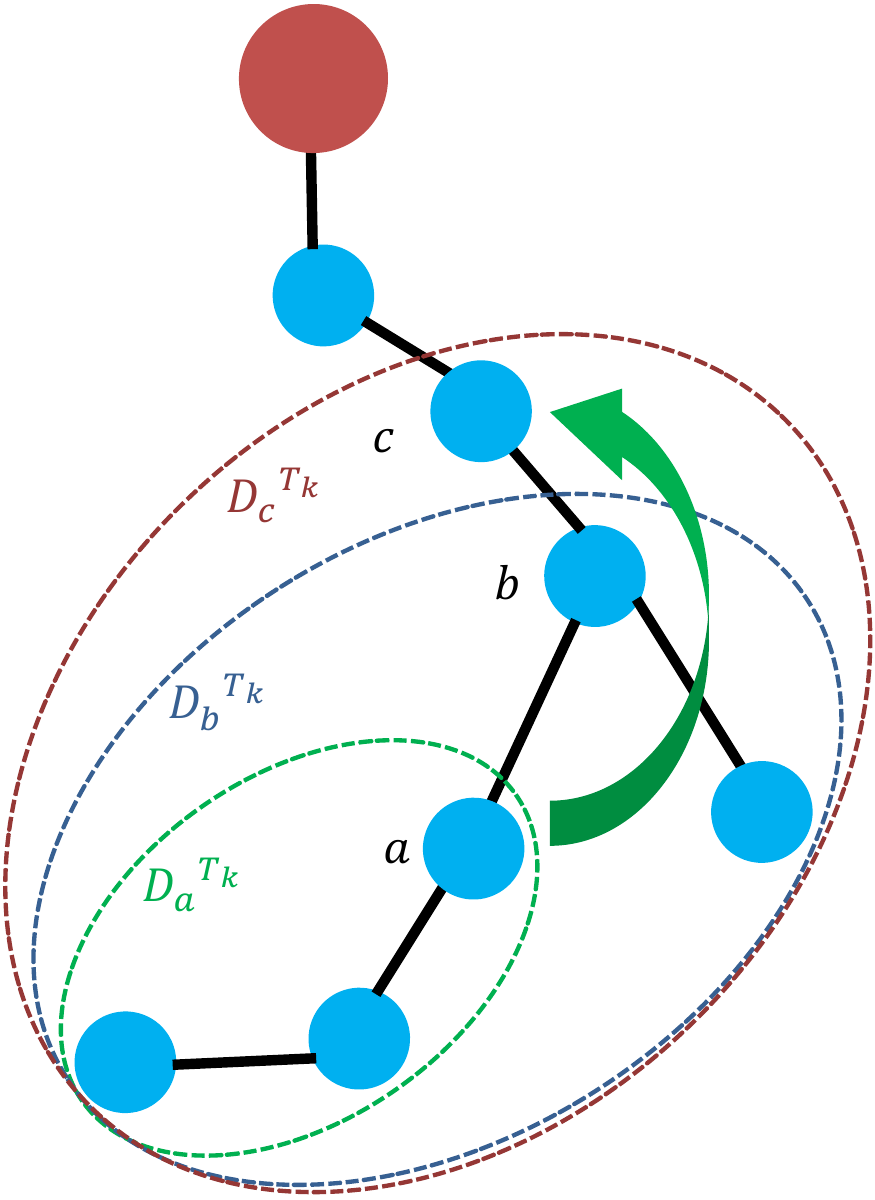}\label{fig:item1}}\hfill
\subfigure[]{\includegraphics[width=0.1559\textwidth]{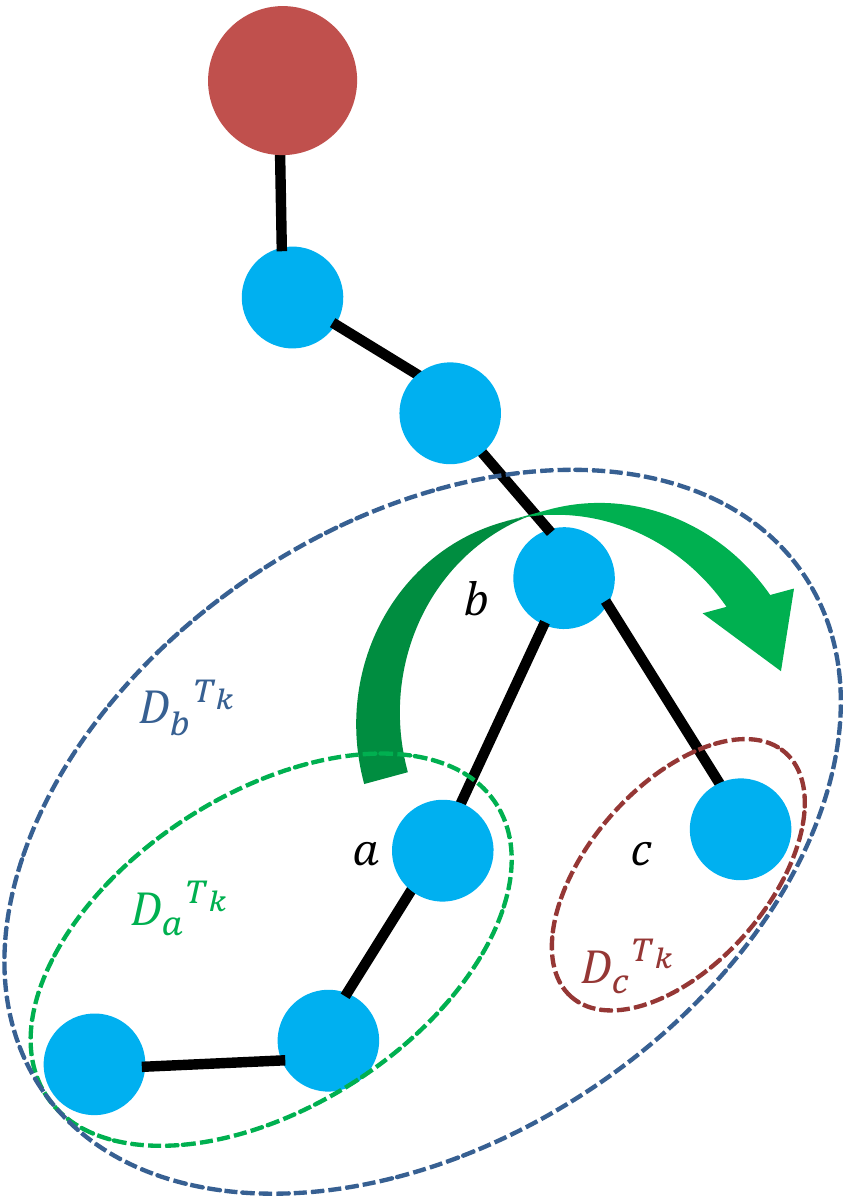}\label{fig:item2}}\hfill
\subfigure[]{\includegraphics[width=0.1625\textwidth]{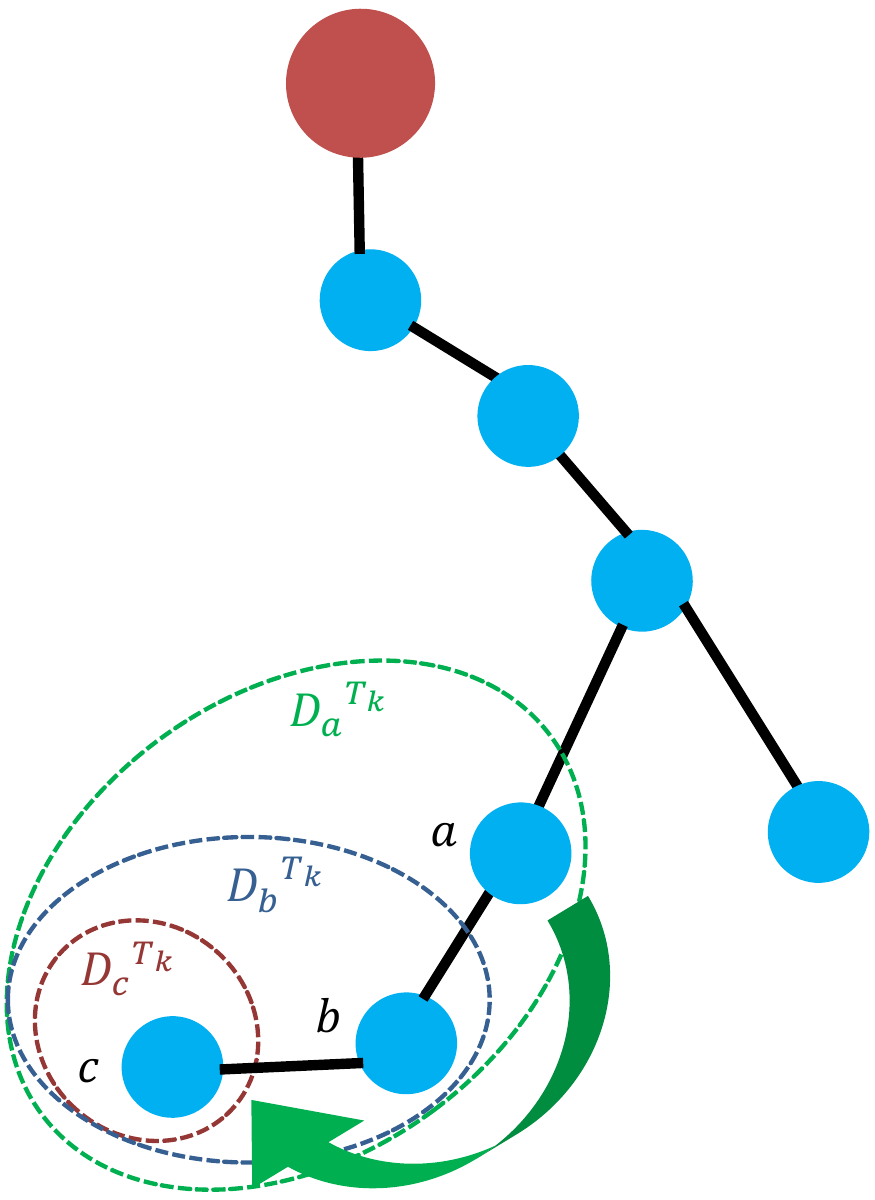}\label{fig:item3}}
\squeezeup
\hspace*{\fill}
\caption{Distribution grid tree with substation/root node represented by large red node. (a) Node $a$ is a descendant of node $b$, node $b$ is a descendant of node $c$. (b) Node $a$ and $c$ are descendants of node $b$ along disjoint sub-trees. (a) Node $c$ is a descendant of node $b$, node $b$ is a descendant of node $a$.
\label{fig:item}}
\end{figure}

\begin{proof}
We give the proof for Case $3$ depicted in Fig.~\ref{fig:item3}. In this case, ${\cal P}^b_{\cal T}- {\cal P}^a_{\cal T} \subseteq {\cal P}^c_{\cal T}-{\cal P}^a_{\cal T}$, where ${\cal P}^a_{\cal T}$ is the set of edges in the unique path from node $a$ to the root node of ${\cal T}$. Further, the sets of descendants of $a,b$ and $c$ satisfy $D^c_{\cal T} \subseteq D^b_{\cal T} \subseteq D^a_{\cal T}$. From Fig.~\ref{fig:item3}, it is clear that any node $d$ belongs to either $D_{\cal T}^c$, $D_{\cal T}^b - D_{\cal T}^c$, $D^a_{\cal T} - D^b_{\cal T}$ or ${\cal V}_{\cal T} - D^a_{\cal T}$. When $d \in D_{\cal T}^c$, using Eq.~(\ref{Hrxinv}), we have,
\begin{align}
H^{-1}_{1/r}(b,d)-H^{-1}_{1/r}(a,d) &= \smashoperator[lr]{\sum_{(ef) \in {\cal P}^b_{\cal T}-{\cal P}^a_{\cal T}}}r_{ef} < \smashoperator[r]{\sum_{(ef) \in {\cal P}^c_{\cal T}-{\cal P}^a_{\cal T}}}r_{ef}\label{f1}\\
\Rightarrow ~H^{-1}_{1/r}(b,d)-H^{-1}_{1/r}(a,d) &< H^{-1}_{1/r}(c,d)-H^{-1}_{1/r}(a,d)\label{first}
\end{align}
For node $d \in D^b_{\cal T} - D^c_{\cal T}$, one derives
\begin{align}
&H^{-1}_{1/r}(b,d)-H^{-1}_{1/r}(a,d) = \smashoperator[lr]{\sum_{(ef) \in {\cal P}^b_{\cal T}-{\cal P}^a_{\cal T}}} r_{ef} ~~~~< \smashoperator[r]{\sum_{(ef) \in {\cal P}^c_{\cal T}\cap {\cal P}^{d}_{\cal T}-{\cal P}^a_{\cal T}} } r_{ef}\label{f2}\\
\Rightarrow ~&H^{-1}_{1/r}(b,d)-H^{-1}_{1/r}(a,d) < H^{-1}_{1/r}(c,d)-H^{-1}_{1/r}(a,d)\label{second}
\end{align}
For $d \in D^a_{\cal T} - D^b_{\cal T}$, one derives
\begin{align}
&H^{-1}_{1/r}(b,d)-H^{-1}_{1/r}(a,d)= \smashoperator[lr]{\sum_{(ef) \in {\cal P}^b_{\cal T}\cap {\cal P}^{d}_{\cal T}-{\cal P}^a_{\cal T}}}r_{ef}~~= \smashoperator[r]{\sum_{(ef) \in {\cal P}^c_{\cal T}\cap {\cal P}^{d}_{\cal T}-{\cal P}^a_{\cal T}} }r_{ef}\label{f3}\\
\Rightarrow~& H^{-1}_{1/r}(b,d)-H^{-1}_{1/r}(a,d) = H^{-1}_{1/r}(c,d)-H^{-1}_{1/r}(a,d)\label{third}
\end{align}
Finally for $d \in {\cal T} - D^a_{\cal T}, H^{-1}_{1/r}(b,d)-H^{-1}_{1/r}(a,d) = H^{-1}_{1/r}(c,d)-H^{-1}_{1/r}(a,d) = 0$. Such inequalities also hold for $H_{1/x}^{-1}$ matrix. Using the inequalities in Eqs.~(\ref{first}, \ref{second},\ref{third}) for $H^{-1}_{1/r}$ and $H^{-1}_{1/x}$ with Eq.~(\ref{usediff_1}) results in $\phi_{ab} < \phi_{ac}$ for Case $3$. The proofs for the other cases ($1$ and $2$) can be done in a similar way and they are thus skipped.
\end{proof}

Further, the following results hold for operational edges in $\cal T$.
\begin{lemma} \label{Lemmacases2}
Let $(ab)$ and $(bc)$ be operational edges in $\cal T$
\begin{enumerate}
\item If node $a$ is the parent of node $b$ (see Fig.~\ref{fig:item3}) then
\begin{align}
\phi_{ab} = \smashoperator[lr]{\sum_{d \in D_{\cal T}^b}}r_{ab}^2\Omega_p(d,d)+x_{ab}^2 \Omega_q(d,d)+2r_{ab}x_{ab}\Omega_{pq}(d,d)\nonumber
\end{align}
\noindent\item If node $b$ is the parent of node $c$ and child of node $a$ (see Fig.~\ref{fig:item3}), then
\begin{align}
\phi_{ac} &= \smashoperator[lr]{\sum_{d \in D_{\cal T}^c}}(r_{ab}+r_{bc})^2\Omega_p(d,d)+(x_{ab}+x_{bc})^2 \Omega_q(d,d)\nonumber\\ &+2(r_{ab}+r_{bc})(x_{ab}+x_{bc})\Omega_{pq}(d,d) \nonumber\\ &+ \smashoperator[lr]{\sum_{d \in D_{\cal T}^b- D_{\cal T}^c}}r_{ab}^2\Omega_p(d,d)+x_{ab}^2 \Omega_q(d,d)+2r_{ab}x_{ab}\Omega_{pq}(d,d)\nonumber\\
&> \phi_{ab} + \phi_{bc} \label{equal1}
\end{align}
\noindent\item If node $b$ is the parent of both nodes $a$ and $c$ (see Fig.~\ref{fig:item2}), then
\begin{align}
\phi_{ac} &= \smashoperator[lr]{\sum_{d \in D_{\cal T}^a}}r_{ab}^2\Omega_p(d,d)\nonumber+x_{ab}^2 \Omega_q(d,d)+2r_{ab}x_{ab}\Omega_{pq}(d,d)\nonumber\\
&+ \smashoperator[lr]{\sum_{d \in D_{\cal T}^c}}r_{bc}^2\Omega_p(d,d)+x_{bc}^2 \Omega_q(d,d)+2r_{bc}x_{bc}\Omega_{pq}(d,d)\nonumber\\
&= \phi_{ab} + \phi_{bc}  \label{equal2}
\end{align}
\end{enumerate}
\end{lemma}
\begin{proof}
\begin{enumerate}
\item We use Eq.~(\ref{Hdiff}) in Eq.~(\ref{usediff_1}) as $(ab)$ is an edge.
\noindent \item We follow the proof in Lemma \ref{Lemmacases}. The result holds as the left sides of Eqs.~(\ref{f1}),(\ref{f2}),(\ref{f3}) here are given by $(r_{ab}+r_{bc})$, $r_{ab}$ and $0$ respectively. The inequality in (\ref{equal1}) is derived by applying Statement $1$ for edges $(ab)$ and $(bc)$ and noting that $(y_1+y_2)*(y_3+y_4)>y_1y_3+y_2y_4$ holds for positive reals $y_1,y_2, y_3,y_4$.
\noindent\item We use the same technique as above. Here $D_{\cal T}^c$ and $D_{\cal T}^a$ are disjoint. Using this fact along with Eq.~(\ref{Hdiff}) for edges $(ab)$ and $(bc)$ results in the equality (\ref{equal2}).
\end{enumerate}
\end{proof}
It is worth mentioning that all three statements in Lemma \ref{Lemmacases2} involve line impedances corresponding to edges $(ab)$ and $(bc)$ only. In the following sections, we use these results to design our topology learning algorithm.

\section{Structure Learning with Full Observation}
\label{sec:algo1}

Our main result for topology learning using voltage magnitude measurements is formulated using Lemma \ref{Lemmacases}.
\begin{theorem}\label{main}
Let the weight of each permissible edge $(ab) \in {\cal E}$ of the original loopy graph be $\phi_{ab} = \mathbb{E}[(\varepsilon_a-\mu_{\varepsilon_a}) -(\varepsilon_b-\mu_{\varepsilon_b})]^2$. Then operational edge set ${\cal E}_{\cal T}$ in radial grid $\cal T$ forms the minimum weight spanning tree of the original graph.
\end{theorem}
\begin{proof}
From Lemma \ref{Lemmacases}, it is clear that for each node $a$, the minimum value of $\phi_{ab}$ along any path in $\cal T$ (towards or away from the root node) is attained at its immediate neighbor $b$ on that path, connected by edge $(ab) \in {\cal E}_{\cal T}$. The minimum spanning tree for the original loopy graph with $\phi$'s as edge weights is thus given by the operational edges in the radial tree.
\end{proof}
Note that if node $a$ is taken as the substation/root node ($\varepsilon_{a}=0$), the weight of any edge $(ab)$  is given by $\phi_{ab} = \Omega_{\varepsilon}(b,b)$. As mentioned in Section \ref{sec:structure}, the substation has one child. In the spanning tree construction, the root is thus connected to the node with lowest variance of voltage magnitude. This is in agreement with Theorem \ref{Theorem1_LC}.

\textbf{Algorithm $1$:} The input consists of voltage magnitude readings for all non-substation buses in the system. An observer computes $\phi_{ab}$ for all permissible edges $(ab) \in {\cal E}$  (including those with the root node) and identifies edges in the minimum spanning tree as the set of operational edges ${\cal E}_{\cal T}$. The root node is restricted to have a single edge. Note that Algorithm $1$ does not need any information on line parameters (resistances and reactances) or on statistics of active and reactive nodal power consumption. If impedances of lines in $\cal E$ and phase angle measurements at all nodes are known, Eqs.~(\ref{PF_LPV_p}), (\ref{means}) and (\ref{volcovar1}) can be used to estimate means and covariances of each node's power injection.

\begin{algorithm}
\caption{Minimum Weight Spanning Tree based Topology Learning}
\textbf{Input:} $m$ voltage magnitude deviations $\varepsilon$ for all nodes, set of all edges $\cal T$.\\
\textbf{Output:} Operational Edge set ${\cal E}_{\cal T}$.
\begin{algorithmic}[1]
\State $\forall (ab) in {\cal E}$, compute $\phi_{ab} = \mathbb{E}[(\varepsilon_a-\mu_{\varepsilon_a}) -(\varepsilon_b-\mu_{\varepsilon_b})]^2$
\State Find minimum weight spanning tree from $\cal E$ with $\phi_{ab}$ as edge weights. Limit degree of substation to $1$.
\State ${\cal E}_{\cal T} \gets $ {edges in spanning tree}
\end{algorithmic}
\end{algorithm}
\textbf{Algorithm Complexity:} Using Kruskal's Algorithm \cite{kruskal1956shortest,Cormen2001}, the minimum spanning tree from $\cal E$ edges can be computed in $O(|{\cal E}|\log|{\cal E}|)$ operations. This is a great improvement over previous iterative or matrix inversion based techniques which scaled as $O(N^3)$, where $N = |{\cal V_{\cal T}}|$ is the number of nodes in the grid. If $\cal E$ is not known or corresponds to the complete graph, Algorithm $1$'s complexity is $O(N^2\log N)$, i.e. it still compares favorably with the prior scheme.

\textbf{Extension to Multiple Trees:} If multiple trees exist in the grid, voltage magnitudes at nodes $a$ and $b$ belonging to disjoint trees will be independent. Thus, $\phi_{ab} = \Omega_{\varepsilon}(a,a) + \Omega_{\varepsilon}(b,b)$. This result can be used to separate nodes into disjoint groups before running Algorithm $1$ to generate the operational tree in each group.

In the next Section, we extend our spanning tree based algorithm to consider cases where information is missing at some fraction of nodes.

\section{Structure Learning with Missing Data}
\label{sec:missing}

In a realistic power grid, communication packet drops or random noise events may erase voltage magnitude measurements for node set ${\cal M}$ in $\cal T$. Following \cite{distgridpart2}, we consider arbitrary placement of unobserved nodes with the following restriction.

\textbf{Assumption $2$:} Missing nodes are separated by greater than two hops in the grid tree $\cal T$.

Note that under assumption $1$, an observable node cannot be connected to two or more unobserved nodes. (We plan to analyze extensions beyond Assumption $2$ in future work.) Additionally, we assume that the adversary estimates or has access to historical information for the values of $\Omega_p, ~\Omega_q$ and $\Omega_{pq}$ covariance matrices for all nodes and impedances of all possible lines in $\cal E$.

\begin{figure}[!bt]
\centering
\hspace*{\fill}
\subfigure[]{\includegraphics[width=0.16\textwidth]{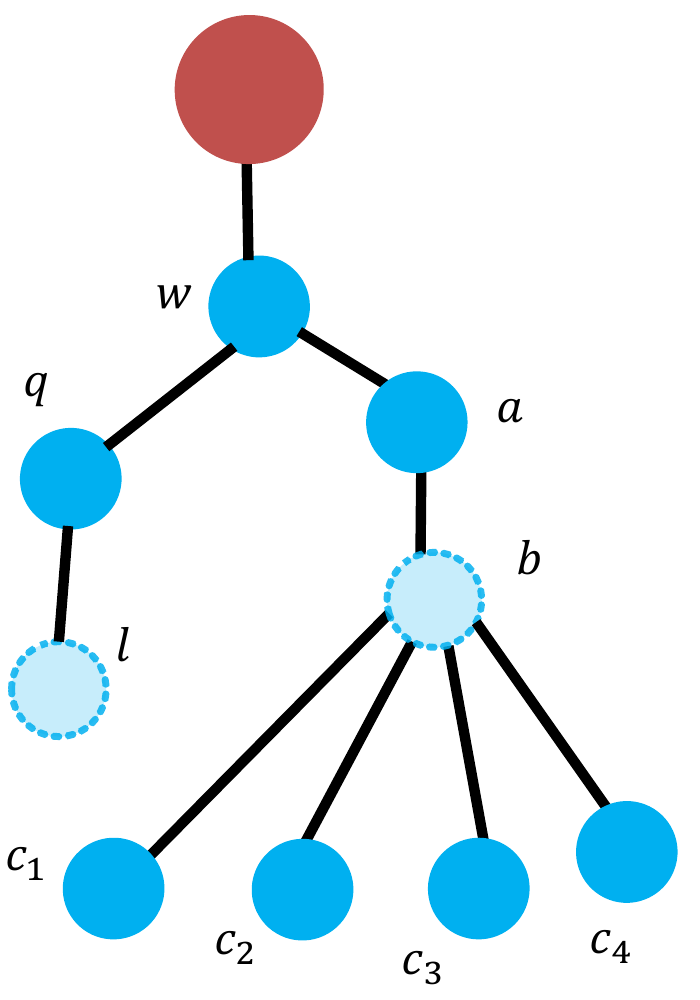}\label{fig:missing1}}\hfill
\subfigure[]{\includegraphics[width=0.16\textwidth]{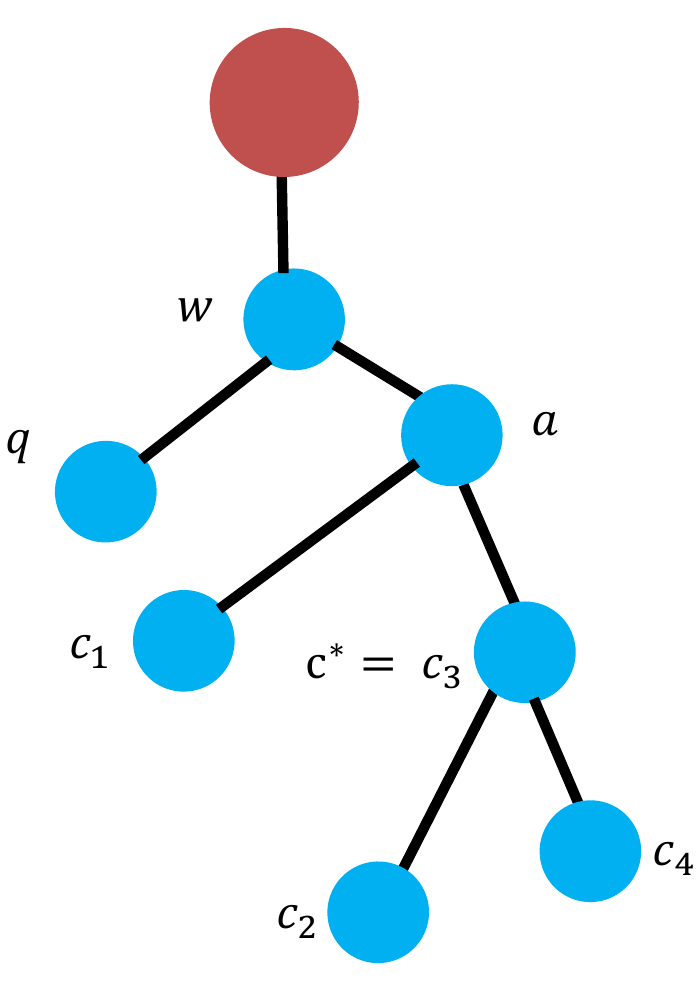}\label{fig:missing2}}\hfill
\subfigure[]{\includegraphics[width=0.16\textwidth]{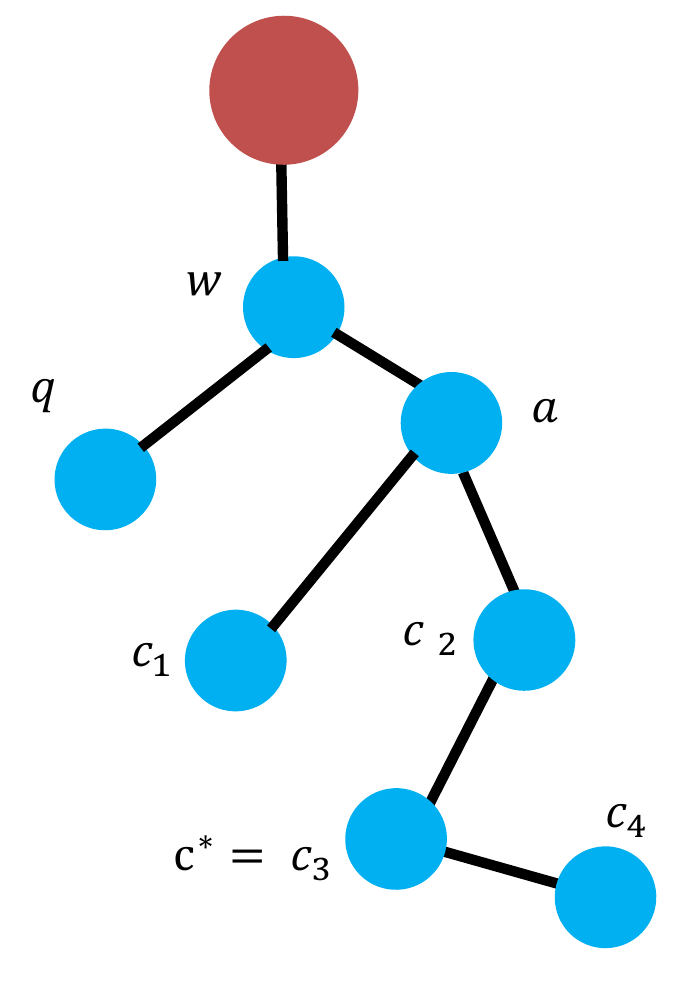}\label{fig:missing3}}
\hspace*{\fill}
\squeezeup
\caption{(a)Distribution grid tree ${\cal T}$ with unobserved leaf node $l$ non-leaf unobserved node $b$. Node $a$ is $b$'s parent while nodes $c_1,c_2,c_3,c_4$ are its children. The spanning tree ${\cal T}_{\cal M}$ of observed nodes exists in either (b) Configuration $A$  or (c) Configuration $B$ as per Theorem \ref{permissiblecases}
\label{fig:missing}}
\end{figure}

To reconstruct operational topology in the presence of missing data, we first construct the minimum weight observable spanning tree ${\cal T}_{\cal M}$ using $\phi_{ab} = \mathbb{E}[(\varepsilon_a-\mu_{\varepsilon_a}) -(\varepsilon_b-\mu_{\varepsilon_b})]^2$ as edge weights between observable nodes. We then analyze edges in tree ${\cal T}_{\cal M }$ and detect unobserved node locations. Consider the situation shown in Fig.~\ref{fig:missing1} where information from the leaf node $l$ is missing. By Assumption $2$, information from its parent ($q$) and grandparent ($w$) are observed in ${\cal T}_{\cal M}$. Note that $\phi_{qw}$ satisfies Statement $1$ in Lemma \ref{Lemmacases2}. If all other descendants of $q$ are known, statement $1$ of the Lemma can be used to identify the existence of unobserved node $l$. 

We now discuss the identification of a non-leaf node with missing information. Assume that information is missing at the node $b$ in Fig.~\ref{fig:missing1}. $b$'s parent $a$ and children node set ${\cal C} = \{c_1, c_2, c_3, c_4\}$ comprise its one-hop neighborhood, and are observable under Assumption $2$. Using Cases $1$ and $3$ in Lemma \ref{Lemmacases}, $\argmin_{d \in D^b_{\cal T} - \{b\}} \phi_{ad} \in {\cal C}$ and $\argmin_{d \in {{\cal V}_{\cal T} - D^b_{\cal T}}} \phi_{c_id}  =a \forall c_i \in {\cal C}$. Thus, descendants of $b$ are connected to the rest of ${\cal T}_{\cal M}$ through edges between its one-hop neighbors (set $C$ and $a$). The following theorem gives the edge configurations possible in ${\cal T}_{\cal M}$ for $a$ and nodes in $\cal C$.

\begin{theorem}\label{permissiblecases}
Let $\argmin_{c_i \in {\cal C}} \phi_{bc_i} = c^*$. No edge $(c_ic_j)$ between children nodes $c_i, c_j \neq c^*$ exists in ${\cal T}_{\cal M}$. All nodes in set ${\cal C}^1= \{c_i\in {\cal C}, \phi_{ac_i} < \phi_{c^*c_i}\}$ are connected to node $a$, while all nodes in ${\cal C}^2 = {\cal C}-{\cal C}^1$ are connected to $c^*$.
\end{theorem}
\begin{proof}
Consider any node pair $c_i,c_j\neq c^*$ in $C$. Using Eq.~(\ref{equal2}) in Lemma \ref{Lemmacases2} and definition of $c^*$, $\phi_{c_ic_j} = \phi_{bc_i}+ \phi_{bc_j} <\phi_{bc_i}+ \phi_{bc^*} = \phi_{c_ic^*}$. Thus, any possible edge between children nodes must include node $c^*$. The edges for each node in sets ${\cal C}^1$ and ${\cal C}^2$ follow immediately by comparing weights with $c^*$ and $a$.
\end{proof} 
Theorem \ref{permissiblecases} does not specify if edge $(ac^*)$ exists in ${\cal T}_{\cal M}$. In fact node $c^*$ will be connected to a node $c\dag \in {\cal C}^1$ instead of $a$ if $\phi_{ac\dag} < \phi_{c^*c\dag} <\phi_{ac^*}$ holds. There are thus two permissible configurations $A$ and $B$ (see  Figs.~\ref{fig:missing2}, \ref{fig:missing3}) in ${\cal T}_{\cal M}$ for connections between one hop neighbors of non-leaf unobservable node $b$. Note that one of sets ${\cal C}^1$ or ${\cal C}^2$ may be empty as well.

Any two nodes in $\cal C$ are children of node $b$ and thus satisfy Statement $3$ in Lemma \ref{Lemmacases2}. Observe that for both configurations $A$ and $B$, this result holds for $c^*$ and any of its children in ${\cal T}_{\cal M}$ that belong to $\cal C$. The result also holds for $c^*$ and its parent in configuration $B$. On the other hand, any node in $\cal C$ and $a$ are actually separated by node $b$ and thus it satisfies Statement $2$ in Lemma \ref{Lemmacases2}. This result thus holds for node $a$ and any of its children from $\cal C$. Statements $2$ and $3$ in Lemma \ref{Lemmacases2} can hence be used to identify unobservable node $b$ in Algorithm $2$. 

\textbf{Algorithm $2$:} Assume that information is missing at the set ${\cal M}$, thus leaving only ${\cal V}_{\cal T} - {\cal M}$ observable. Covariance matrices for power injection at all nodes of the observed set are assumed known to the observer along with impedances of all lines in $\cal E$. Algorithm $2$, first, constructs spanning tree ${\cal T}_{\cal M}$ for observed nodes using edge weights for all node combinations given by $\phi$. Observed nodes in ${\cal T}_{\cal M}$ are then arranged in reverse topological order (decreasing depth from root node). This is done as unobserved node locations are iteratively searched from leaf sites inward towards the root (see Step \ref{step1}). For each leaf $b$ with parent $a$, Steps \ref{step2} to \ref{step3} checks if edge $(ab) \in {\cal E}_{\cal T}$ with or without some unobserved leaf node $h$ connected to $b$. For undecided nodes in $C$, the Algorithm first checks for configuration $A$ or $B$ described in the preceding discussion. Step \ref{step4} determines if nodes in $C$ and $a$ are separated by a unobserved node $h$ using Statement $2$ in Lemma \ref{Lemmacases2}. If such a node doesn't exist, Step \ref{step5} search for a unobserved node that is parent of both nodes in $C$ and node $a$ using using Statement $3$ in Lemma \ref{Lemmacases2}. Nodes $a$ and set $C$ are removed from the observed tree ${\cal T}_{\cal M}$ in each iteration and discovered edges are added to ${\cal E}_{\cal T}$. Further, injection covariances at the recently identified descendants are added for use in later checks involving results from Lemma \ref{Lemmacases2}. Note that only in the final case (Step \ref{step5}), the unobserved node $h$ is not removed from set $M$ as its parent node has not been determined yet. This process is iterated by picking a new node $a$ with all children as leaf nodes until no nodes with missing information remain to be discovered.

\textbf{Complexity:} Computing the spanning tree for observed nodes has complexity $O((N-|{\cal M}|)^2\log(N-|{\cal M}|))$. Sorting observed nodes in topological order is done in linear time ($O(N-|{\cal M}|)$) \cite{Cormen2001}. Finally, checking (Steps \ref{step1}, \ref{step2}, \ref{step3}, \ref{step4}) for all iterations has complexity $O((N- |{\cal M}|)|{\cal M}|)$ as total observed nodes and edges number $O((N- |{\cal M}|))$ and searching over unobserved nodes takes at most $|{\cal M}|$ steps. The overall complexity of Algorithm $2$ is thus $O((N-|{\cal M}|)^2\log(N-|{\cal M}|)+(N- |{\cal M}|)|{\cal M}|)$ which is $O(N^2\log N)$ in the worst case. Note that this is also the worst-case complexity of Algorithm $1$.

\textbf{Relation to Learning Probabilistic Graphical Model:} It is worth noting that in the tree-structured GM learning \cite{choi2011learning}, edge $(ac^*)$ always exists due to the graph-additivity of edge weights and configuration $B$ in Fig.~\ref{fig:missing3} is not realized. The inequality in Eq.~(\ref{equal1}) of Lemma \ref{Lemmacases2} shows that $\phi$ may be strictly increasing with the number of graph hops and thus it does not satisfy graph additivity in general. Non-additivity of edge weights makes our topology learning approach a generalization of the additive model in \cite{choi2011learning} .

\textbf{Extensions:} We briefly mention two extensions of Algorithm $2$, planning to analyze these in details in the future. First, Algorithm $2$ can be used for structure learning \emph{when injection covariances at unobserved nodes are not known}. Here each unobserved node must have at least two children for unique identification. Second, Algorithm $2$ will be extended to operate \emph{when unobserved nodes are separated by $2$ hops}. In this case, permissible configurations in addition to $A$ and $B$ (see Fig.~\ref{fig:missing}) need to be checked. A modification of Statement $2$ in Lemma \ref{Lemmacases2} will be used to detect unobserved nodes. In the following Section, we discuss the performance of our designed algorithms through experiments on test networks.
\begin{algorithm*}
\caption{Minimum Weight Spanning Tree based Topology learning with Missing Data}
\textbf{Input:} Injection covariances $\Omega_p, \Omega_q, \Omega_{pq}$ of all nodes, Missing nodes Set ${\cal M}$, $m$ voltage deviation observations $\varepsilon$ for nodes in ${\cal V}_{\cal T} -{\cal M}$, set of all edges $\cal T$ with line impedances.\\
\textbf{Output:} Operational Edge set ${\cal E}_{\cal T}$.
\begin{algorithmic}[1]
\State $\forall$ observable nodes $a,b$, compute $\phi_{ab} = \mathbb{E}[(\varepsilon_a-\mu_{\varepsilon_a}) -(\varepsilon_b-\mu_{\varepsilon_b})]^2$
\State Find minimum weight spanning tree ${\cal T}_{\cal M}$ with $\phi_{ab}$ as edge weights. Limit degree of substation to $1$.
\State Sort nodes in ${\cal T}_{\cal M}$ in reserve topological order.
\While {$|{\cal M}| >0$}
\State Select node $a$ whose children set ${\cal C}$ in ${\cal T}_{\cal M}$ consists only of leaf nodes \label{step1}
\ForAll{$b \in {\cal C}$}
\If {$\phi_{ab}$ satisfy Statement $1$ in Lemma \ref{Lemmacases2} with $D^b_{\cal T}  =  \{b\}$}\label{step2}
\State ${\cal E}_{\cal T} \gets {\cal E}_{\cal T} \cup \{(ab)\}$, ${\cal C} \gets {\cal C}-\{b\}$, Add injection covariance of $b$ to $a$. Remove node $b$ from ${\cal T}_{\cal M}$.
\EndIf
\If {$\exists h \in {\cal M}$ s..t. $\phi_{ab}$ satisfy Statement $1$ in Lemma \ref{Lemmacases2} with $D^b_{\cal T}  =  \{b,h\}$}
\State  ${\cal E}_{\cal T} \gets {\cal E}_{\cal T} \cup \{(ab), (bh)\}$, ${\cal M} \gets {\cal M}-\{h\}$, ${\cal C} \gets {\cal C}-\{b\}$, Add injection covariance of $b$ and $h$ to $a$. Remove node $b$ from ${\cal T}_{\cal M}$.
\EndIf\label{step3}
\EndFor
\If {$|{\cal C}| > 0$}
\If{$\exists b \in {\cal C}, h \in {\cal M}$ s..t. $\phi_{ab}$ satisfy Statement $2$ in Lemma \ref{Lemmacases2} with $D^b_{\cal T}  = \{b\}$ and $D^h_{\cal T}  = \{h\}\cup {\cal C}$} \label{step4}
\State ${\cal E}_{\cal T} \gets {\cal E}_{\cal T} \cup \{(ah)\} \cup \{(ch) \forall c  \in {\cal C}\}$, ${\cal M} \gets {\cal M}-\{h\}$, ${\cal C} \gets \emptyset$, Add injection covariances $\forall c \in {\cal C}$ and $h$  to $a$. Remove nodes in ${\cal C}$ from ${\cal T}_{\cal M}$.
\Else
\State Pick $b \in {\cal C}$. Find $h \in {\cal M}$ s..t. $\phi_{ab}$ satisfy Statement $3$ in Lemma \ref{Lemmacases2} with $h$ as parent and $D^b_{\cal T}  = \{b\}$ , $D^a_{\cal T}  = \{a\}$. \label{step5}
\State ${\cal E}_{\cal T} \gets {\cal E}_{\cal T} \cup \{(ah)\} \cup \{(ch) \forall c  \in {\cal C}\}$, ${\cal C} \gets \emptyset$, Add injection covariances of $a$ and $\forall c \in {\cal C}$ to $h$. Remove $a$ and nodes in ${\cal C}$ from ${\cal T}_{\cal M}$.
\EndIf
\EndIf
\EndWhile
\end{algorithmic}
\end{algorithm*}

\section{Experiments}
\label{sec:experiments}

Here we demonstrate performance of Algorithm $1$ in determining the operational edge set ${\cal E}_{\cal T}$ of the radial grid ${\cal T}$. We consider a radial network \cite{testcase2,radialsource} with $29$ load nodes and one substation as shown in Fig.~\ref{fig:case}. In each of our simulation runs, we first collect complex power injection samples at the non-substation nodes from a multivariate Gaussian distribution that is uncorrelated between different nodes as per Assumption $1$. We use LC-PF model to generate nodal voltage magnitude measurements. Finally, we introduce $30$ additional edges (at random) forming the loopy edge set ${\cal E}$. The additional edges are given random impedances comparable to those of operational lines. We, first, test performance of the Algorithm $1$ for the case where locations of edges in the set $\cal E$ and voltage magnitude measurements at all non-substation nodes are available. We show results for topology learning for this case in Fig.~\ref{fig:plotadjerrors}. Note that the estimation is extremely accurate and average errors expressed relative to the size of the operational edge set) decay to zero at the sample sizes less than $50$. We also estimate covariance matrices of complex nodal power injections using the just reconstructed radial operating topology and plot results in Fig.~\ref{fig:plotcoverrors}. For covariance estimation, line impedances of the set $\cal E$ and samples of phase angle measurements are used along with voltage magnitude samples as input. The relative errors in this case decay exponentially with increase in the number of the measurement samples.

\begin{figure}[!bt]
\centering
\includegraphics[width=0.36\textwidth,height =.30\textwidth]{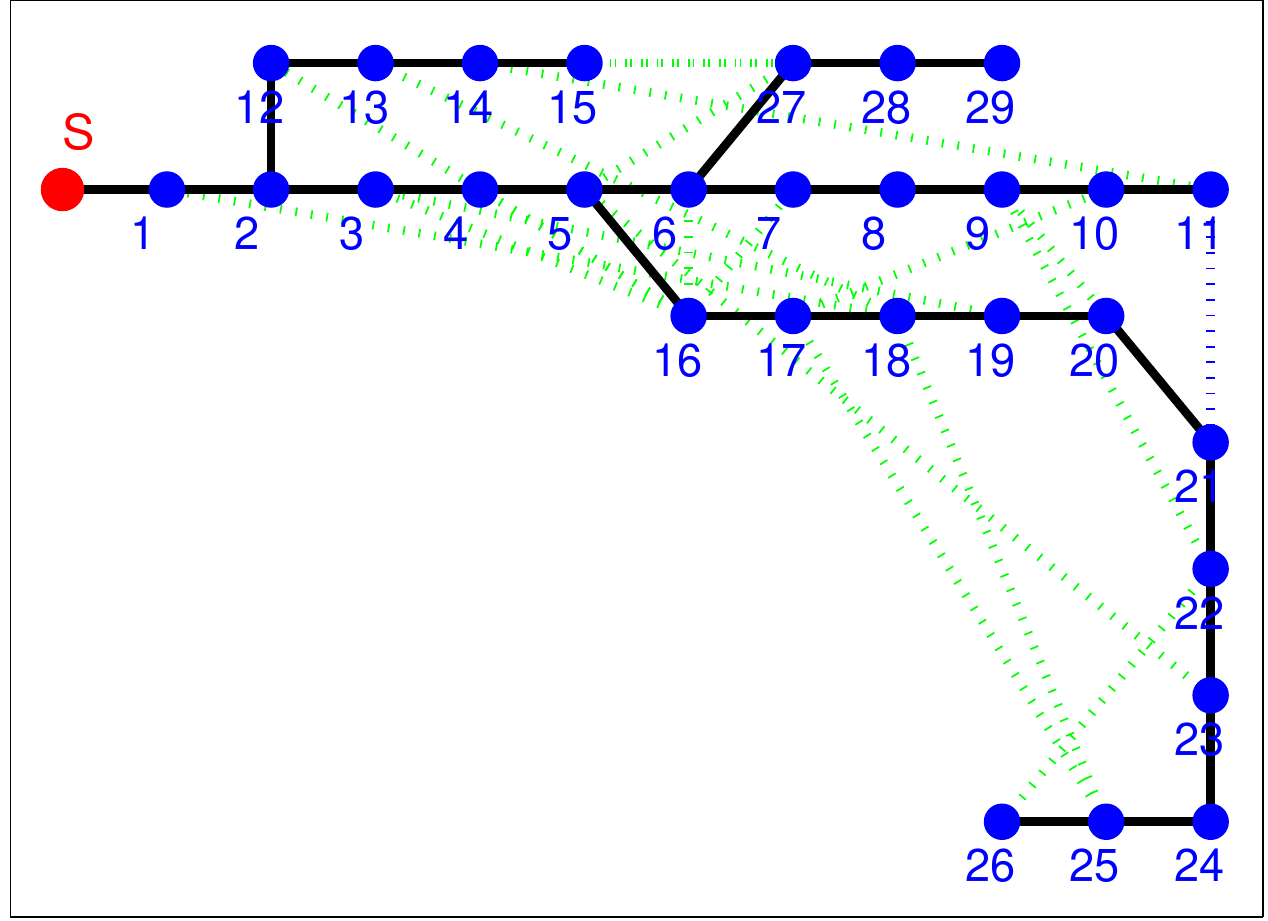}
\vspace{-.10cm}
\caption{Layouts of the grids tested. The red circle represents substation (marked as $S$). The blue circles represent numbered load nodes. Black lines represent operational edges. The additional open lines are represented by dotted green lines.}
\label{fig:case}
\vspace{-2mm}
\end{figure}

\begin{figure}[!bt]
\centering
\subfigure[]{\includegraphics[width=0.42\textwidth,height = .35\textwidth]{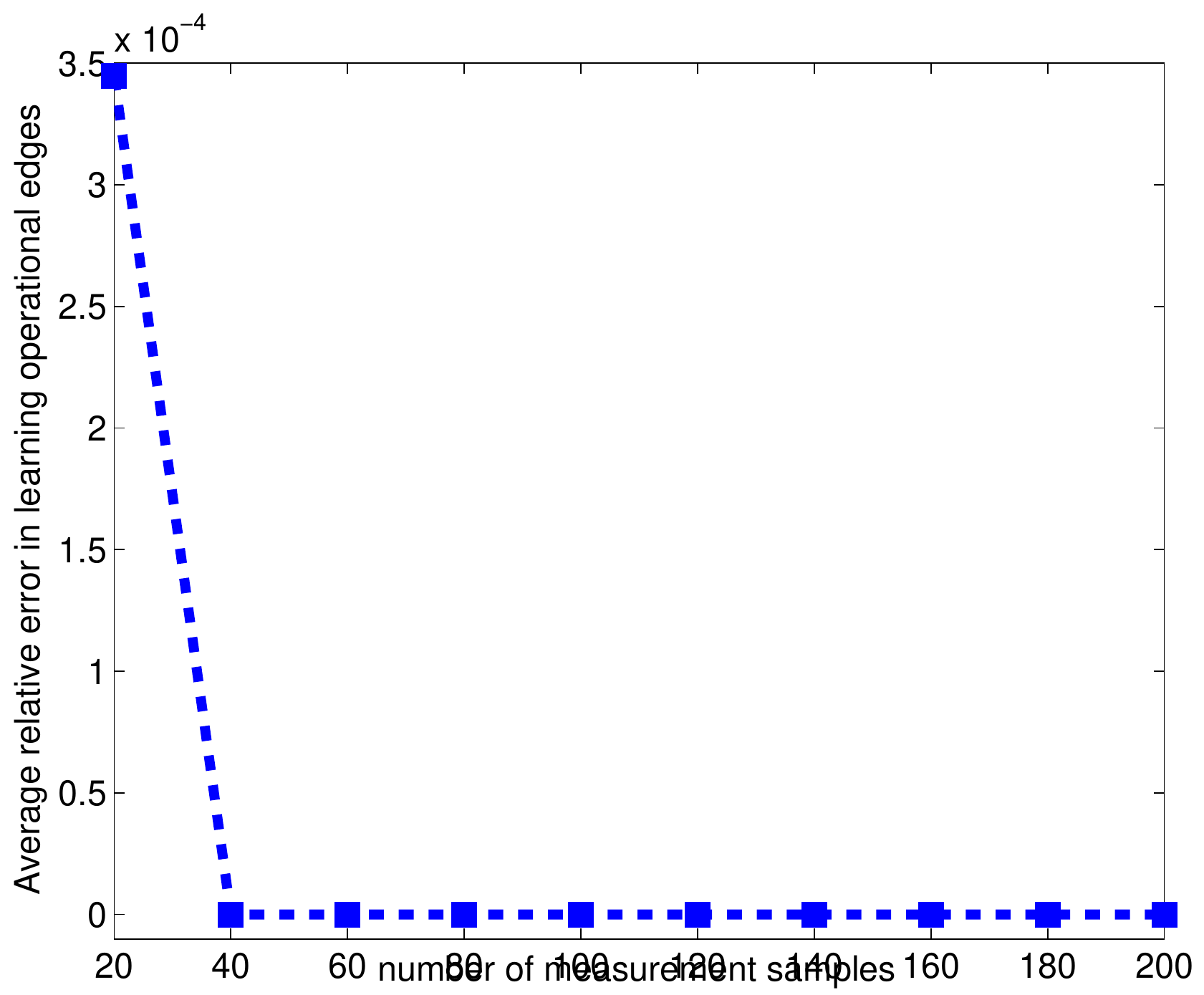}\label{fig:plotadjerrors}}
\subfigure[]{\includegraphics[width=0.42\textwidth,height = .35\textwidth]{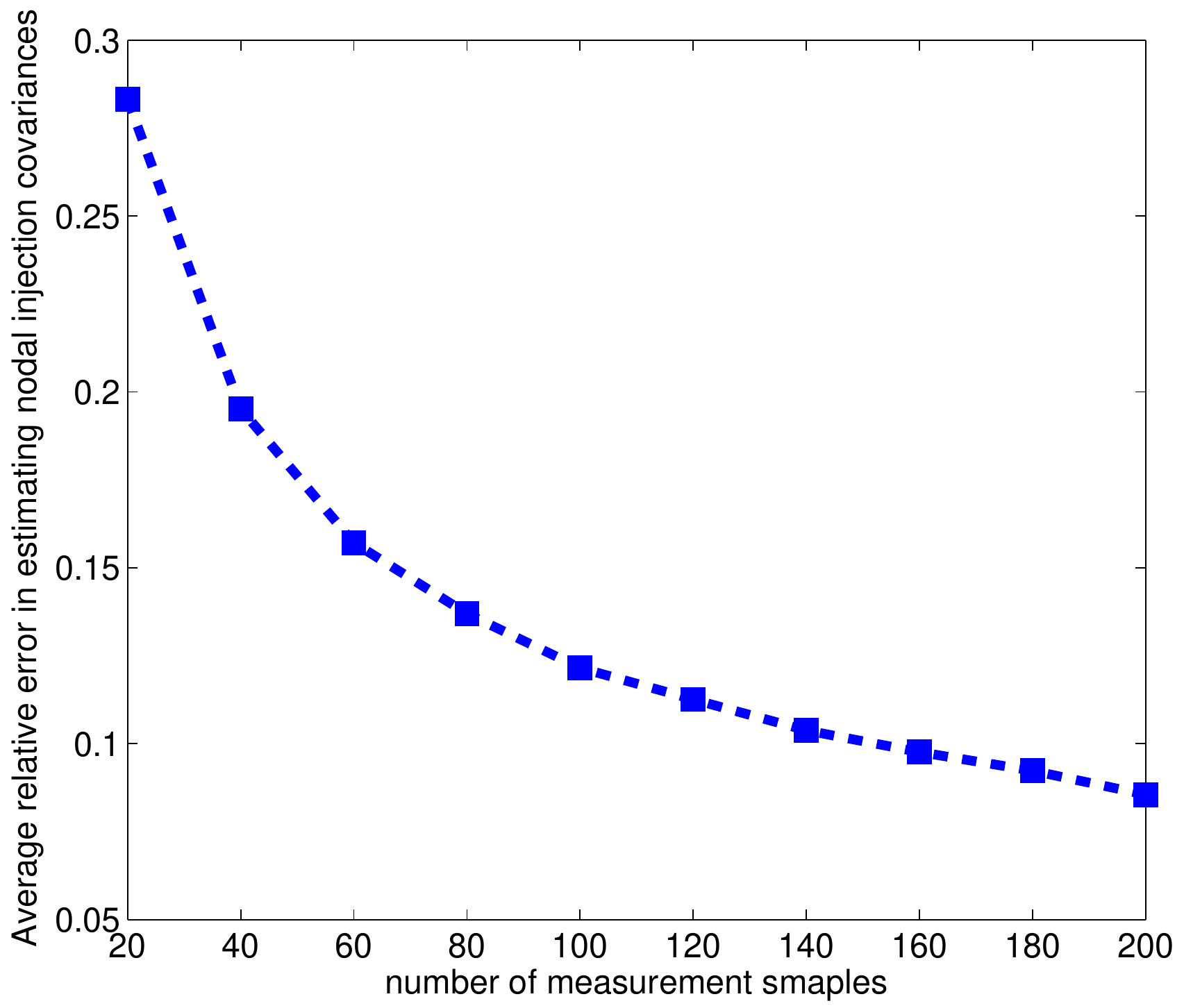}\label{fig:plotcoverrors}}
\squeezeup
\caption{Average fractional errors vs number of samples used in Algorithm $1$ for (a) Learning operational edges (c) Estimating nodal injection covariances.
\label{fig:algo1}}
\end{figure}

Next, we present simulations for Algorithm $2$ where the operational grid structure is reconstructed in the presence of unobserved nodes. We consider three cases with information at the nodes $4$, $6$ and $8$ missing. The location of the unobserved nodes are selected at random in accordance with Assumption $2$. Voltage magnitudes at the unobserved nodes are removed from the input data. Covariance of power injections at all the load nodes and impedances of all the lines within the loopy edge set $\cal E$ are provided as input to the observer. The average number of errors shown in Fig.~\ref{fig:plotmissing} decreases steadily with increase in the number of samples. This tendency is seen clearly for all the cases of the unobserved node sets. Further, the average errors increase with increase in the number of unobserved nodes for a fixed number of measurement samples. The average errors produced by Algorithm $2$ are significantly lower in comparison with the respective algorithm from \cite{distgridpart2}, however (and as expected) the Algorithm is significantly less accurately than Algorithm $1$ where all nodes are observed.
\begin{figure}[!bt]
\centering
\subfigure[]{\includegraphics[width=0.42\textwidth,height = .37\textwidth]{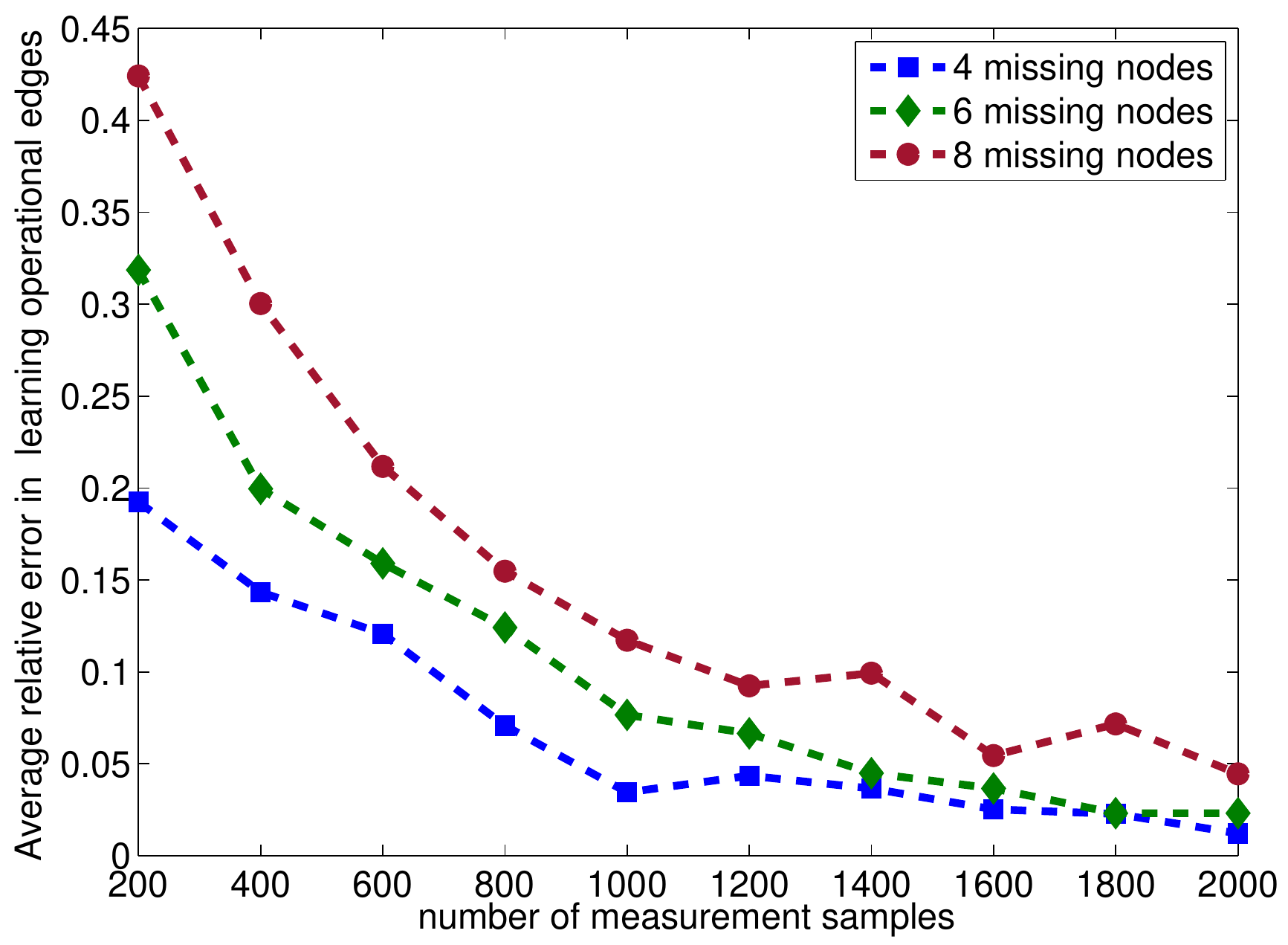}\label{fig:plotmissing}}
\squeezeup
\caption{Average fractional errors in learning operational edges vs number of samples used in Algorithm $2$ with missing data. Information is missing at the nodes $4, 6$ and $8$.
\label{fig:algo2}}
\end{figure}

\section{Conclusions}
\label{sec:conclusions}

Identifying the operational edges in the distribution grids is critical for real-time control and reliable management of different grid operations. In this paper, we study the problem of learning the radial operating structure from a dense loopy grid graph. Under an LC (linear coupled) power flow model, we show that if edge weights between load nodes are defined as the variance of the difference of their voltage magnitudes, the minimum weight spanning tree optimization over the loopy physical layout outputs operational radial structure. Using this spanning tree property, we design a fast structure learning algorithm that uses only nodal voltage magnitude measurements for the input. We then extend the spanning tree based framework to learn the operational structure when available voltage measurements are limited to a subset of the grid nodes. For unobserved nodes separated by greater than three hops, the learning algorithm is able to identify locations of the missing measurements by verifying properties of our voltage magnitude based edge weights. In this case, statistics of nodal injections and line impedances are used as a part of the input. We demonstrate good performance of the learning algorithm through experiments on distribution grid test cases. Finally, we discuss how voltage magnitude based edge weights in our algorithm generalizes edge metrics used in learning schemes of probabilistic GMs. In future we plan to generalize our approach reducing restrictions, e.g. allowing unobserved nodes to be separated by less than two hops and utilizing less information about nodal consumption.

\bibliographystyle{IEEETran}
\bibliography{../../Bib/FIDVR,../../Bib/SmartGrid,../../Bib/voltage,../../Bib/trees}
\end{document}